\documentclass{amsart}
\usepackage{amsmath, amsthm, amssymb}
\usepackage{amssymb}
\usepackage{amsfonts}
\usepackage{amscd}
\usepackage{mathrsfs}
\usepackage{latexsym}
\usepackage{amsmath}
\usepackage{txfonts}
\usepackage{tikz}
\usepackage{graphicx}
\usepackage{subfig}

\usepackage[all, cmtip]{xy}

\newtheorem{theorem}{Theorem}[section]
\newtheorem{lemma}[theorem]{Lemma}
\newtheorem{corollary}[theorem]{Corollary}
\newtheorem{proposition}[theorem]{Proposition}

\theoremstyle{definition}
\newtheorem{definition}[theorem]{Definition}

\theoremstyle{remark}
\newtheorem{remark}[theorem]{Remark}

\newcommand{\supp}[1]{\textrm{supp}(#1)}

\newcommand{\sshf}[1]{\mathcal{O}_{#1}}
\newcommand{\shf}[1]{\mathcal{#1}}

\newcommand{\prj}[1]{\mathbb{P}^{#1}}

\newcommand{\iso}{\simeq}
\newcommand{\ses}[3]{0\rightarrow#1\rightarrow#2\rightarrow#3\rightarrow{0}}

\newcommand{\rk}[1]{\textrm{rank}(#1)}
\newcommand{\is}[1]{\mathscr{I}_{#1}}
\newcommand{\lsp}[1]{\left\langle{#1}\right\rangle}

\newcommand{\paren}[1]{\left(#1\right)}

\numberwithin{equation}{section}

\begin{document}
\allowdisplaybreaks
\title{On the singularities of effective loci of line bundles}

\author{Lei Song}

\address{Department of Mathematics, Statistics and Computer Science, University of Illinois at Chicago, Chicago, IL 60607-7045, USA}
\email{lsong4@uic.edu}




\dedicatory{}


\begin{abstract}
We prove that every irreducible component of semi-regular loci of effective line bundles in the Picard scheme of a smooth projective variety has at worst rational singularities. This generalizes Kempf's result on rational singularities of $W^0_d$ for smooth curves. We also work out an example of such loci for a ruled surface.
\end{abstract}

\maketitle

\section{Introduction}
Fix a ground field $k$, which is algebraically closed and of characteristic 0. Let $X$ be a smooth projective curve of genus $g$.  For $r, d\ge 0$, the Brill-Noether locus is defined as
\begin{equation*}
 \textrm{supp}(W^r_d(X))=\{L\in\textrm{Pic}(X)\;|\; h^0(L)\ge r+1, \deg(L)=d\}.
\end{equation*}
There has been extensive research on these loci in literature (cf. \cite{ACGH}). A special kind of Brill-Noether loci is $W^0_d(X)$, which is the image of the
Abel-Jacobi map $\varphi: X_d\rightarrow\textrm{Pic}^d X$, where $X_d$ denotes the $d$th symmetric product of $X$ and $\textrm{Pic}^d X$ denotes the Picard variety of degree $d$ line bundles on $X$. When $d=g-1$, $W^0_{g-1}$ is a theta divisor if $\textrm{Pic}^{g-1}X$ is indentified with the Jacobian of the curve.

Kempf \cite{KE} proved that $W^0_d(X)$ has only rational singularities, so in particular it is Cohen-Macaulay and normal; and for $1\le d\le g-1$, the tangent cone over a point $[L]\in W^0_d(X)$ admits a rational resolution from the normal bundle of the fibre $F=\varphi^{-1}([L])$. He also computed the degree of the tangent cone, generalizing Riemman's formula on multiplicity of theta divisors. A celebrated generalization of Kempf's theorem in the case $d=g-1$ is due to Ein and Lazarsfeld \cite[Theorem 1]{Ein97} stating that any principal polarization divisor $\Theta\subset A$ on an abelian variety is normal and has rational singularities.

This paper attempts to extend part of Kempf's results on $W^0_d$ for curves to higher dimensional varieties using the approach and technique of Ein \cite{Ein}, where the author studied the normal sheaf $\shf{N}$ of the fibre $F$. He showed $\shf{N}$ can be reconstructed from the multiplication map $H^0(\sshf{F}(1))\otimes H^1(\shf{N}(-1))\rightarrow H^1(\shf{N})$, and proved that for a general curve $X$, $\shf{N}\iso\rho\sshf{F}\oplus (H^1(X, L)\otimes\Omega_{F}(1))$, where $\rho=g-(r+1)(g+r-d)$ is the Brill-Noether number. A large part of his results were built on a locally free resolution of $\shf{N}^*$.

Now let $X$ be a smooth projective variety of arbitrary dimension. Let $\textrm{Pic}(X)$ and $\textrm{Div}(X)$ denote the Picard scheme and divisor scheme, which parameterize line bundles and effective divisors on $X$ respectively. One still has the Abel-Jacobi map $\varphi: \textrm{Div}(X)\rightarrow\textrm{Pic}(X)$, where $\textrm{Div}(X)$ plays the same role as $X_d$. However, as a closed subscheme of the Hilbert scheme $\textrm{Hilb}(X)$, $\textrm{Div}(X)$ may be very singular.
Even for $\dim X=2$, an example due to Severi and Zappa in 1940s shows that $\textrm{Div}(X)$ can be nonreduced. For this reason, we restrict ourselves to those so called semi-regular line bundles (see \S 2 for definition),
and consider the semi-regular locus $W^0_{\textrm{sr}}(X)$ they form in $\textrm{Pic}(X)$. We refer to \cite{KL} for background of $\textrm{Pic}(X)$ and $\textrm{Div}(X)$. Our main theorem is

\begin{theorem}\label{main theorem}
Let $X$ be a smooth projective variety. Then any irreducible component of $W^0_{\textrm{sr}}(X)$ has only rational singularities.
\end{theorem}

When $X$ is a curve, $W^0_{\textrm{sr}}=\coprod_{d\ge 0} W^0_d$ and $W^0_d$ is irreducible, and so the theorem recovers Kempf's result that $W^0_d$  has rational singularities.

The  paper is organized as follows: in \S 2, we study the conormal sheaf of fibres of the Abel-Jacobi map. We derive a resolution of the sheaf and obtain several interesting consequences. With a criterion of rational singularities based on Kov\'{a}cs's work, we prove our main theorem. In \S 3, one example of an irreducible component $W^0_{\textrm{sr}}$ of a ruled surface is analyzed in detail. In the appendix we prove an auxiliary result on varieties swept out by linear spans of divisors of linear systems on an embedded curve.

{\it Acknowledgments}:
The author is very grateful to his advisor Lawrence Ein for suggesting this problem and many helpful discussions. He would like to thank Chih-Chi Chou, Izzet Coskun for stimulating discussions concerning \S 3, and thank Wenbo Niu and Pete Vermeire for answering his questions. He also would like to thank the anonymous referee, whose careful reading and valuable suggestions greatly improve this paper.

\section{Rational singularities of $W^0_{\text{sr}}(X)$}

\subsection{Semi-regular line bundles and their loci}
Let $X$ be a smooth projective variety.
It is well known that $\textrm{Pic}(X)$ is separated and smooth over $k$. As $\textrm{Hilb}(X)$ breaks into connected components according to Hilbert polynomials, so does $\textrm{Pic}(X)$. Fix an ample line bundle $\sshf{X}(1)$ on $X$. For each line bundle $L$ on $X$, there exists a $\mathbb{Q}$-coefficient polynomial $P_{L}$ such that $P_{L}(n)=\chi(L(n))$ for $n\in\mathbb{Z}$. The $P_{L}$ is constant over any connected component of $\textrm{Pic}(X)$.
The Abel-Jacobi map $\varphi: \textrm{Div}(X)\rightarrow \textrm{Pic}(X)$, which sends an effective divisor $D$ to the associated line bundle $\sshf{X}(D)$, is a projective morphism. For any line bundle $L$ on $X$, canonically $\varphi^{-1}([L])\iso |L|$, where $[L]$ is the corresponding point of $L$ in $\textrm{Pic}(X)$ (cf. \cite{KL}).

\begin{definition}\label{semi-regular line bundle}
An effective Cartier divisor $D$ on $X$ is \textit{semi-regular} if the boundary map
\begin{equation*}
    \partial: H^1(\sshf{D}(D))\rightarrow H^2(\sshf{X})
\end{equation*}
is injective.
A line bundle $L$ is \textit{semi-regular} if $L$ is effective, and $D$ is semi-regular for all $D\in|L|$.
\end{definition}

\begin{remark}
If $X$ is a curve, then all effective divisors, line bundles are automatically semi-regular.
The reader can check that a necessary condition for $L$ to be semi-regular is that $h^1(L)\le q$ (see Corollary \ref{Clifford inequality}), and sufficient conditions are either $h^1(L)=0$ or $H^1(\sshf{D}(D))=0$ for all $D\in |L|$. The second one is however rather strong. For instance, when $X$ is a surface and $p_g=h^0(\omega_X)>0$, $H^1(\sshf{D}(D))=0$ implies that $\supp D\subset \textrm{Bs}(|\omega_X|)$.
\end{remark}

\begin{theorem}[Severi-Kodaira-Spencer]\label{Severi-Kodaira-Spencer}
Assume $\text{char}(k)=0$. $\textrm{Div}(X)$ is smooth at $[D]$ of the expected dimension
\begin{equation}\label{expected dimension}
    R:=h^0(\sshf{X}(D))-h^1(\sshf{X}(D))-1+h^1(\sshf{X})
\end{equation}
if and only if $D$ is semi-regular (cf. \cite{KL} or \cite{MU}).
\end{theorem}

\begin{definition}[semi-regular locus]
\begin{equation*}
\textrm{supp}(W^0_{\textrm{sr}}(X))=\{L\in\textrm{Pic}(X)\;|\;  L \;\textrm{is effective and semi-regular} \}.
 \end{equation*}
\end{definition}

\begin{remark}
From Theorem \ref{Severi-Kodaira-Spencer}, we see that ``semi-regular" is an open condition on the locus of effective line bundles $W^0(X)\subset\textrm{Pic}(X)$ and that any connected component of $W^0_{sr}(X)$ is irreducible. Thus each component of $W^0_{\textrm{sr}}(X)$ is a subvariety of $\textrm{Pic}(X)$. It is worth noting that not every irreducible component of $W^0(X)$ contains some semi-regular line bundle, see remark (\ref{non semi-regular locus}) for an example.
\end{remark}

Next we explain the idea of the proof of Theorem \ref{main theorem}. Given $[L]\in\Omega$, which is a component of $W^0_{\textrm{sr}}$, there exists a unique $\Delta_0$ among all irreducible components $\{\Delta_i\}$ of $\textrm{Div}(X)$, such that $\dim{\Delta_0}=R$ and $\varphi^{-1}([L])\subset (\Delta_0)_{reg}\backslash \cup_{i\neq 0} \Delta_i$, where $(\Delta_0)_{reg}$ is the regular locus of $\Delta_0$.
Consider the induced Abel-Jacobi morphism $\varphi: \Delta_0\rightarrow \Omega$. By properness of $\varphi$, there is a smooth neighborhood $U$ of $\varphi^{-1}([L])$ inside $\Delta_0$ such that $\varphi(U)$ is open in $\Omega$ and $\varphi^{-1}\varphi(U)=U$. By abuse of notation, we denote this $U$ by $\textrm{Div}(X)$, hence the  normal sheaf of the fibre by $\shf{N}_{\varphi^{-1}([L]) /{\textrm{Div}(X)}}$ instead of $\shf{N}_{\varphi^{-1}([L])/ U}$, and simply by $\shf{N}$ if the fibre is clear from the context. Under the semi-regularity assumption, $\shf{N}$ can be calculated from the universal family of divisors associated to $|L|$ by base change theorem. It turns out that the vector bundle $\shf{N}^*$ (on the projective space $|L|$) has Castelnuovo-Mumford regularity 0, therefore formal function theorem shows that $R^i\varphi_*(\sshf{\Delta_0})_{[L]}=0$ for $i>0$. Finally, though $\varphi$ is not a resolution of singularities of $\Omega$, as it is not birational in general, Theorem \ref{Kovacs} guarantees that $\Omega$ has at worst rational singularities at $[L]$.

\subsection{A criterion for rational singularities}
The theorem below is a characterization of rational singularities. The original assumption is more general than what we state here.
\begin{theorem}[Kov\'{a}cs \cite{SK}]\label{Kovacs}
Let $f: Y\rightarrow X$ be a surjective proper morphism of varieties. Assume that $Y$ has rational singularities and that $f_*\sshf{Y}\iso \sshf{X}$, $R^i f_*\sshf{Y}=0$ for all $i>0$. Then $X$ has rational singularities.
\end{theorem}

Based on the above, we obtain Theorem \ref{rational}, which is more general than what we actually need to prove Theorem \ref{main theorem} and may be applied to other problems.

\begin{lemma}\label{proper line bundle}
Let $X$ be a smooth projective variety of dimension $d$ with $H^1(X, \sshf{X})=0$, and $L$ a globally generated ample line bundle on $X$ with the property that
$K_X+(d-1)L$ is noneffective. Then any $m$-regular (with respect to $L$) coherent sheaf $\shf{F}$ admits a locally free resolution of the form:
\begin{equation*}
    \cdots\rightarrow V_i\otimes L^{-(m+i)}\cdots\rightarrow V_1\otimes L^{-(m+1)}\rightarrow V_0\otimes L^{-m}\rightarrow\shf{F}\rightarrow 0,
\end{equation*}
where each $V_i$ is a finite dimensional vector space. Consequently if $\shf{F}$ is locally free and $0$-regular, then any symmetric power of $\shf{F}$ is $0$-regular.
\end{lemma}
\begin{proof}
By Kodaira vanishing and the assumption $H^1(X, \sshf{X})=1$, the condition that $K_X\otimes L^{d-1}$ is noneffective implies that reg(${\sshf{X}})\le 1$. Then the result follows from \cite{LA} Remark 1.8.16.
\end{proof}
\begin{remark}
Notice that the existence of such $L$ in Lemma \ref{proper line bundle} imposes a strong restriction on $X$: $-K_X$ is big. Examples for Fano varieties are  $\prj{d}$, quadric hypersurface $Q\subset\prj{d+1}$, and $\mathbb{P}(\sshf{\prj{1}}^{\oplus (d-1)}\oplus\sshf{\prj{1}}(1))$ (with Fano index $1$). If $-K_X$ is also nef, the assumption $H^1(X, \sshf{X})=0$ is redundant by Kawamata-Viehweg vanishing.
\end{remark}

\begin{theorem}\label{rational}
Let $f: Y\rightarrow X$ be a projective morphism from a smooth variety $Y$ onto a normal variety $X$. Let $p\in X$ be a closed point. Suppose
the scheme-theoretic fiber $F$ is a smooth variety of dimension $d$ with $H^i(F, \sshf{F})=0$ for all $i>0$, and the conormal sheaf $\shf{N}^* _{F/Y}$ is 0-regular with respect to a globally generated ample line bundle $L$ on $F$, such that $K_F+(d-1)L$ is noneffective. Then $X$ has rational singularities in a neighbourhood of $p$.
\end{theorem}

\begin{proof}
Consider the Stein factorization of $f: Y\xrightarrow{f'} Y'\xrightarrow{g} X$, where $f'$ is projective with connected fibers, and $g$ is a finite morphism. Then $Z:=g^{-1}(p)$ is a reduced closed point, for otherwise $F=f'^{-1}(Z)$ would be nonreduced. Therefore $g$ is generically one to one map, and hence birational. Since $X$ is normal, we have $g$ is an isomorphism, and hence $f$ has connected fibres.

Let $\is{}$ be the ideal sheaf of $F$ in $Y$. By Lemma \ref{proper line bundle}, any symmetric power $S^n(\shf{N}^*_{F/Y})\iso\is{}^n/{\is{}^{n+1}}$ is $0$-regular. In particular, all its higher cohomologies vanish. From the exact sequence
\begin{equation*}
0\rightarrow \is{}^n/\is{}^{n+1}\rightarrow \mathcal{O} _{(n+1)F}\rightarrow \mathcal{O} _{nF}\rightarrow 0,
\end{equation*}
we get $H^i({O} _{(n+1)F})\cong H^i({O} _{nF})$ for all $i,n>0$.
Then we conclude that
$H^i({O} _{nF})=0$ for all $i>0$.
So $R^i f_*(\mathcal{O} _Y)^{\wedge}_p=0$ for all $i>0$, by the formal function theorem (cf. \cite[III 11.1]{HA}). Since the support of the coherent sheaf $R^i f_*(\mathcal{O} _Y)$ is closed, by shrinking $X$, we can assume $R^i f_*(\mathcal{O} _Y)=0$ on X for $i>0$.

It's clear that $f_*\mathcal{O}_Y=\mathcal{O}_X$, since $X$ is normal and fibers are connected.
At this point, we apply Theorem \ref{Kovacs} to conclude the proof.
\end{proof}

\subsection{Conormal sheaf of the fibre of $\varphi$}\mbox{}\\

In the rest of \S 2, $L$ stands for a line bundle on $X$ with $r=\dim{|L|}$ and $b=h^1(X, L)$, $F$ denotes the fibre $\varphi^{-1}([L])\iso |L|$ of the Abel-Jacobi map $\varphi: \textrm{Div}(X)\rightarrow\textrm{Pic}(X)$. Let $\frak{m}$ be the maximal ideal of $\sshf{\textrm{Pic}(X), [L]}$, $\bar{\frak{m}}$ the maximal ideal of $\sshf{W^0_{\textrm{sr}}, [L]}$, and $\is{}$ the ideal sheaf of $F$ in $\textrm{Div}(X)$.

Given an effective line bundle $L$, $|L|\iso \mathbb{P}(H^0(X, L)^*)$. Let $Y=X\times|L|$ and $p, q$ be the two projections. By K$\ddot{\textrm{u}}$nneth formula, $\Gamma(Y, p^*L\otimes q^*\sshf{}(1))\iso H^0(X, L)\otimes H^0(X, L)^*$. Fix a basis of the vector space $H^0(X, L)$, say $x_0, \cdots, x_r$. The canonical section $\displaystyle {s=\sum^r_{i=0} x_i\otimes x^*_i}$ defines a relative Cartier divisor of $Y$ over $|L|$ via
\begin{equation}\label{universal family}
    0\rightarrow{\sshf{Y}}\xrightarrow{.s}{p^*L\otimes q^*\sshf{}(1)}\rightarrow{\sshf{\mathscr{D}}(\mathscr{D})}\rightarrow 0.
\end{equation}

Denote the two induced projections from $\mathscr{D}$ to $X$ and $|L|$ also by $p$ and $q$. The divisor $\mathscr{D}$ is actually an incidence correspondence in the sense: for any $x\in X$, $p^{-1}(x)$ parameterizes the effective divisors passing through $x$; for any $[D]\in |L|$, $q^{-1}([D])$ is precisely the divisor $D$.
\vspace{0.3cm}

By the universal property of $\textrm{Div}(X)$, there is a unique morphism $j: |L|\rightarrow \textrm{Div}(X)$ such that $\mathscr{D}=j^*\mathscr{U}$, where $\mathscr{U}$ is the universal divisor over $\textrm{Div}(X)$, see the Cartesian diagram below. In fact $j$ is a closed immersion.
$$\xymatrix{
D \ar@{^{(}->}[d] \ar@{^{(}->}[r]  &\mathscr{D} \ar@{^{(}->}[d] \ar@{^{(}->}[r]      & \mathscr{U}\ar@{^{(}->}[d]\\
X \ar[d] \ar@{^{(}->}[r]           &X\times|L| \ar[d]^{q} \ar@{^{(}->}[r]^{j'}        & X\times\textrm{Div}(X) \ar[d]^{\pi}\\
[D]\ar@{^{(}->}[r]                 &|L| \ar@{^{(}->}[r]^{j}                          & \textrm{Div}(X)}$$

Suppose $L$ is semi-regular, then there is a smooth open neighborhood $V$ of $|L|$ in $\text{Div}(X)$, and hence $\shf{T}_\text{Div(X)}\big|_V$, the restriction of the tangent sheaf of $\text{Div}(X)$ to $V$, is locally free.

Denoting the projections from $X\times\textrm{Div}(X)$ by $\pi'$ and $\pi$ respectively, we have the natural morphism
\begin{equation}\label{natural morphism}
    \shf{T}_\text{Div(X)}\big|_{V}\iso \pi_*\paren{\pi^*\shf{T}_\text{Div(X)}\big|_{V}} \hookrightarrow \pi_*\paren{\pi^*\shf{T}_\text{Div(X)}\big|_{V}\oplus {\pi'}^*\shf{T}_X}\rightarrow \paren{\pi_*\sshf{\mathscr{U}}(\mathscr{U})}\big|_V.
\end{equation}
On the other hand, by Grauert's theorem, for all $[D]\in V$,
\begin{equation}\label{iso at stalks}
    \pi_*\sshf{\mathscr{U}}(\mathscr{U})\otimes\kappa([D])\iso H^0(X, \shf{N}_{D/X})\iso\shf{T}_\textrm{Div(X)}\otimes\kappa([D]),
\end{equation}
where $\kappa([D])$ is the residue field at $[D]$ and the last isomorphism follows from the property of $\pi$.

In view of (\ref{natural morphism}) and (\ref{iso at stalks}), we get an isomorphism
\begin{equation*}
    \paren{\pi_*\sshf{\mathscr{U}}(\mathscr{U})}\big|_V\iso\shf{T}_\textrm{Div(X)}\big|_V.
\end{equation*}

Then the base change theorem (cf. \cite[Lecture 7]{MU}) implies that
\begin{equation*}
    j^*\pi_*\sshf{\mathscr{U}}(\mathscr{U})\iso q_*{j'}^*\sshf{\mathscr{U}}(\mathscr{U}).
\end{equation*}

Since ${j'}^*\sshf{\mathscr{U}}(\mathscr{U})\iso\sshf{\mathscr{D}}(\mathscr{D})$, one has
\begin{equation}\label{pull back of normal sheaf}
    q_*{\sshf{\mathscr{D}}(\mathscr{D})}\iso\shf{T}_\text{Div(X)}\big |_{|L|}.
\end{equation}

The key theorem below is a generalization of \cite[Theorem 1.1]{Ein}.

\begin{theorem}\label{conormal sheaf}
With notation as above, let $L\in W^0_{sr}$ and $\shf{N}^*$ be the conormal sheaf of $F$ in $\textrm{Div}(X)$. Then there is an exact sequence
\begin{equation}\label{resolution of conormal sheaf}
    \ses{H^1(L)^*\otimes\sshf{F}(-1)}{H^1(\sshf{X})^*\otimes\sshf{F}}{\shf{N}^*}.
\end{equation}
\end{theorem}
\begin{proof}
Note $\sshf{Y}(\mathscr{D})\iso p^*L\otimes q^*\sshf{F}(1)$. Applying $q_*$ to (\ref{universal family}), we get the exact sequence
\begin{equation*}
    0\rightarrow \sshf{F}\rightarrow H^0(L)\otimes\sshf{F}(1)\rightarrow {\shf{T}_{Div(X)}\big|_{F}}\rightarrow H^1(\sshf{X})\otimes\sshf{F}\rightarrow H^1(L)\otimes\sshf{F}(1)\rightarrow 0,
\end{equation*}
where the third term comes from (\ref{pull back of normal sheaf}). The surjectivity of the last map is for the reason as follows.
For any point $[D]\in F$, we have the commutative diagram
$$\xymatrix{
R^1q_*\sshf{Y}\otimes\kappa([D])\ar[d] \ar[r]^{\iso} &H^1(\sshf{X})\ar[d]\\
R^1q_*\sshf{Y}(\mathscr{D})\otimes\kappa([D])\ar[r]^{\iso} &H^1(\sshf{X}(D))}$$

The two horizontal maps are isomorphisms because of Grauert's theorem. Since $\partial: H^1(\sshf{D}(D))\rightarrow H^2(\sshf{X})$ is injective by the semi-regularity assumption, the right vertical map is surjective, so is the left one.

Since the cokernel of $\sshf{F}\rightarrow H^0(L)\otimes\sshf{F}(1)$ is $\shf{T}_{F}$, we get the short sequence
\begin{equation*}
   \ses{\shf{N}}{H^1(\sshf{X})\otimes\sshf{F}}{H^1(L)\otimes\sshf{F}(1)}.
\end{equation*}
Dualizing it, we get
\begin{equation*}
    \ses{H^1(L)^*\otimes\sshf{F}(-1)}{H^1(\sshf{X})^*\otimes\sshf{F}}{\shf{N}^*}\qedhere
\end{equation*}
\end{proof}

\begin{corollary}\label{higher cohomology of conormal sheaf}
$\shf{N}^*$ has Castelnuovo-Mumford regularity 0.\qed
\end{corollary}

\subsection{Proof of theorem \ref{main theorem}}\mbox{}\\
From now on, let $q$ denote $h^1(X, \sshf{X})$, the irregularity of $X$.
\begin{lemma}\label{alpha map}
The natural map $\bigoplus_{n\ge 0}\bar{\frak{m}}^n/\bar{\frak{m}}^{n+1}\rightarrow\bigoplus_{n\ge 0}H^0(\is{}^n/\is{}^{n+1})$ is a surjective graded $k$-algebra morphism. Furthermore, if $R\le q$, then it is an isomorphism.
\end{lemma}
\begin{proof}
First consider the commutative diagram of $k$-vector spaces:
$$\xymatrix{
\frak{m}/{\frak{m}}^2 \ar@{->>}[r]\ar[d]^{\iso} & \bar{\frak{m}}/\bar{\frak{m}}^{2}\ar[d]\\
H^1(\sshf{X})^* \ar@{->>}[r] & H^0\paren{\shf{N}^*_{F/{Div(X)}}}}$$
Clearly the top horizontal map is surjective. The left vertical map is an isomorphism by \cite[Theorem 5.11]{KL}. The bottom horizontal one is surjective by (\ref{resolution of conormal sheaf}). Therefore
\begin{equation*}
    \bar{\frak{m}}/\bar{\frak{m}}^{2}\rightarrow  H^0\paren{\shf{N}^*_{F/{Div(X)}}}
\end{equation*}
is surjective.
For $n> 1$, consider the commutative diagram:
$$\xymatrix{
S^n(\bar{\frak{m}}/\bar{\frak{m}}^2) \ar@{->>}[r]\ar@{->>}[d] & \bar{\frak{m}}^n/\bar{\frak{m}}^{n+1}\ar[d]\\
S^nH^0(\is{}/\is{}^{2} )\ar@{->>}[r] & H^0(\is{}^n/\is{}^{n+1})}$$
The bottom horizontal map is surjective, because $\is{}/\is{}^{2}=\shf{N}^*_{F/{Div X}}$ is 0-regular. It follows that the right vertical map is surjective.

A proof for isomorphism when $R\le q$ can be found in \cite{Ein} Proposition 3.1 (c) and Theorem 3.2.
\end{proof}

\begin{proof}[Proof of Theorem 1.1]

By Theorem \ref{rational} and Corollary \ref{higher cohomology of conormal sheaf}, to finish the proof, it remains to show that any irreducible component $\Omega$ of $W^0_{\textrm{sr}}$ is normal. Let $Y=\varphi^{-1}(\Omega)\subset\textrm{Div}(X)$. Consider the commutative diagram with exact rows
$$\xymatrix{
0\ar[r] & \bar{\frak{m}}^n/ \bar{\frak{m}}^{n+1}\ar[d]^{\alpha_n}\ar[r] & {\sshf{\Omega, [L]}/{\bar{\frak{m}}}^{n+1}}\ar[d]^{\beta_{n+1}}\ar[r] &\sshf{\Omega, [L]}/\bar{\frak{m}}^n\ar[r]\ar[d]^{\beta_n} &0\\
0\ar[r] & H^0(\is{}^n/\is{}^{n+1})\ar[r] & H^0(\sshf{(n+1)F})\ar[r] & H^0(\sshf{nF})\ar[r]& 0}$$
By Lemma \ref{alpha map}, $\alpha_n$'s are surjective. By Snake lemma and induction on $n$, we get $\beta_n$ is surjective for all $n\ge 1$. It follows that $\sshf{\Omega, [L]}^{\wedge}=\varprojlim \sshf{\Omega, {[L]}}/\bar{\frak{m}}^n\twoheadrightarrow\varprojlim H^0(\sshf{nF})\iso(\varphi_*\sshf{Y})^{\wedge}_{[L]} $ by formal function theorem. Thus the canonical morphism $\sshf{\Omega, {[L]}}^{\wedge}\rightarrow (\varphi_*\sshf{Y})^{\wedge}_{[L]}$ is an isomorphism. Since the completion is a fully faithful functor, we get that $\varphi_*\sshf{Y}\iso\sshf{\Omega}$. Since $Y$ is smooth and all fibres are connected, $\Omega$ is normal.
\end{proof}

\subsection{Some consequences of Theorem \ref{conormal sheaf}}\mbox{}\\

The corollary below is a generalization of Clifford theorem to higher dimensional varieties. It would be interesting to study when the equalities can be achieved.
\begin{corollary}\label{Clifford inequality}
Assume $[L]\in W^0_{\textrm{sr}}$ and $h^1(L)>0$. Then\\
(i) $h^0(L)+h^1(L)\le q+1$.\\
(ii) If $X$ is a projective surface, then $h^0(L)\le \frac{\chi(L)+q+1}{2}$.
\end{corollary}
\begin{proof}
By \cite[Proposition 2.5]{Ein}, the shape of the resolution (\ref{resolution of conormal sheaf}) of $\shf{N}^*$ forces $\textrm{rank}(\shf{N}^*)\ge r$.
Since $\dim{\textrm{Div(X)}}=R$, one has $\textrm{rank}(\shf{N}^*)=R-r$. Recall $R=h^0(L)-h^1(L)+q-1$ in (\ref{expected dimension}), we get $(i)$.
If $X$ is a surface, then $R\le \chi(L)+q-1$. So $h^0(L)=r+1\le \frac{R}{2}+1\le\frac{\chi(L)+q+1}{2}$.
\end{proof}

\begin{corollary}
Let $[L]\in W^0_{\textrm{sr}}$ and assume $R\le q$. Then up to a constant, the Hilbert-Samuel function $\psi$ for $\sshf{W^0, [L]}$ is
\begin{eqnarray*}
  \psi(p)   &=& \left\{\begin{array}{ll}
                   {p+q-1\choose q} & \textrm{if $b\le r$},  \\
                   {p+q-1\choose q}+\sum^b_{i=r+1}(-1)^{i+r}{p+q-i-1\choose q}{b\choose i}{i-1\choose r} & \textrm{if $b>r$}.
                       \end{array}\right.
\end{eqnarray*}
\end{corollary}
\begin{proof}
The exact Eagon-Northcott complex associated to (\ref{resolution of conormal sheaf}) is
\begin{eqnarray*}
    0&\rightarrow & S^{p-b}H^1(\sshf{X})^*\otimes\bigwedge^bH^1(L)^*\otimes\sshf{\prj{r}}(-b)\rightarrow\cdots\rightarrow S^{p-1}H^1(\sshf{X})^*\otimes H^1( L)^*\otimes\sshf{\prj{r}}(-1)\\
    & \rightarrow & S^pH^1(\sshf{X})^*\otimes \sshf{\prj{r}}\rightarrow S^p\shf{N}^*\rightarrow 0.
\end{eqnarray*}
So for $p\gg 0$,
\begin{eqnarray*}
  \chi(S^p\shf{N}^*) &=& \sum^b_{i=0}(-1)^i\chi\paren{S^{p-i}H^1(\sshf{X})^*\otimes\bigwedge^iH^1(L)^*\otimes\sshf{\prj{r}}(-i)} \\
        &=& \left\{\begin{array}{ll}
                   {p+q-1\choose q-1} & \textrm{if $b\le r$},  \\
                   {p+q-1\choose q-1} +\sum^b_{i=r+1}(-1)^{i+r}{p+q-i-1\choose q-1}{b\choose i}{i-1\choose r} & \textrm{if $b>r$}.
                       \end{array}\right.
\end{eqnarray*}
Since
\begin{eqnarray*}
  \Delta\psi(p) &=&  \psi(p+1)-\psi(p)\\
   &=& \dim{(\bar{\frak{m}}^p/{\bar{\frak{m}}^{p+1}})} \\
   &=& h^0(S^p\shf{N}^*) \hspace{3cm}\textrm{by (\ref{alpha map})}\\
   &=& \chi(S^p\shf{N}^*) \hspace{3.1cm}\textrm{$S^p\shf{N}^*$ is 0-regular}
\end{eqnarray*}
we get the conclusion by the proof of \cite[I, 7.3 (b)]{HA} .
\end{proof}
The multiplicity $\mu(\sshf{W, [L]})$ is defined as $(\text{leading coefficient of } \psi)\cdot(\deg{\psi})!$. To avoid combinatorial relations for calculating $\mu$, we resort to intersection theory.

\begin{corollary}\label{Riemann-Kempf}
Let $[L]\in W^0_{\textrm{sr}}$ and assume that $\varphi$ is birational. Then $\mu={b\choose r}$.
\end{corollary}

\begin{proof}
Since $\mu$ coincides with top Segre class of $([L], W_{sr}^0)$, which is invariant under a birational proper morphism (cf. \cite[Chap. 4]{FUL}),
\begin{eqnarray*}
   \mu &=& s_0([L], W_{sr}^0) \\
       &=& s_{r}(\prj{r}, \textrm{Div}X)\\
       &=& (-1)^rs_r(\shf{N}^*).
\end{eqnarray*}
Let $H$ be the class of a hyperplane section in $\prj{r}$, again by (\ref{resolution of conormal sheaf}),
\begin{equation*}
    s_t(\shf{N}^*)=(1-Ht)^b,
\end{equation*}
which concludes the proof.
\end{proof}

\section{Examples: Ruled Surface}
Though the condition for semi-regular line bundles (Definition \ref{semi-regular line bundle}) looks quite strong, we shall show that it does not automatically imply that $W^0_{sr}(X)$ is smooth, which is clear in the case that $X$ is a curve. We shall construct an example of dimension 2 (for simplicity) following such rules: (i) $W^0_{sr}(X)$ is singular; (ii) $W^0_{sr}(X)$ is nontrivial, i.e.~$W^0_{sr}(X)$ is not isomorphic to $W^0(C)$ for some curve $C$ and $\dim{W^0_{sr}(X)}\ge 2$, in particular $q=\dim{\text{Pic}(X)}\ge 2$; (iii) $W^0_{sr}(X)$ can be explicitly computed, at least for one component.  This motivates the work in this section.
\vspace{0.1cm}

We start by fixing some notations.
\begin{itemize}
  \item $C$: a smooth projective curve of genus $g\ge 2$.
  \item $B$: a line bundle on $C$ of degree $d_0\ge 3$.
  \item $E$: a rank two vector bundle on $C$, fitting into the exact sequence $\ses{\sshf{C}}{E}{B}$.
  \item $X$: $=\mathbb{P}(E)$ and $\pi: X\rightarrow C$ the canonical projection.
  \item $\textrm{Pic}^{(i, j)}(X)$: the connected component of $\textrm{Pic}(X)$ consisting of line bundles of the form $\sshf{X}(i)\otimes\pi^* M$, where $\deg{M}=j$.
  \item $W^0_{i, j}(X):= W^0_{\textrm{sr}}(X)\cap \textrm{Pic}^{(i, j)}(X)$, where ``sr'' is omitted for simplicity.
  \item Linear span: suppose $C$ is embedded into a projective space by a very ample line bundle $A$. Let $D$ be an effective divisor on $C$. One has the exact sequence $\ses{H^0(C, A(-D))}{H^0(C, A)}{Q}$. The \textit{linear span} $\lsp{D}:=\mathbb{P}(Q)\subseteq \mathbb{P}(H^0(C, A))$. Note $H^0(C, A(-D))$ is the space of linear defining equations of $\lsp{D}$.
  \item $X_{|L|}$: the variety swept out by linear spans of divisors in $|L|$, where $L$ is a line bundle on $C$. See appendix for its basic properties. The notion $X_{|L|}$ depends on the choice of embedding of $C$.
\end{itemize}

\subsection{Geometric interpretation of the extension $\ses{\sshf{C}}{E}{B}$}\mbox{}\\

Since $\deg(K_C\otimes B)=2g-2+d_0\ge 2g+1$, the complete linear system $|K_C\otimes B|$ induces an embedding
\begin{equation*}
    \varphi: C\hookrightarrow\prj{N}=\mathbb{P}\paren{H^0\paren{K_C\otimes B}},
\end{equation*}
where $N=g+d_0-2$.
By Serre duality,
\begin{equation*}
    H^0(C, K_C\otimes B)^*\iso H^1\paren{C, B^{-1}}\iso \textrm{Ext}^1\paren{B, \sshf{C}}.
\end{equation*}
So a point $\eta\in \prj{N}$ determines an extension of $B$ by $\sshf{C}$
\begin{equation}\label{extension}
    \ses{\sshf{C}}{E}{B},
\end{equation}
uniquely up to isomorphism.

\begin{remark}
The idea of realizing an extension class of $B$ by $\sshf{C}$ as a point in $\mathbb{P}(H^0(K_C\otimes B))$ is borrowed from Bertram \cite{BE}.
\end{remark}

\subsection{Characterization of $W^0_{1, \star}(X)$}
\begin{proposition}[Effectiveness Criterion]\label{characterization}
Let $M$ be a line bundle on $C$. $H^0(X, \sshf{X}(1)\otimes\pi^* M)=H^0(C, E\otimes M)\neq 0$ if and only if
\begin{enumerate}
  \item either $H^0(M)\neq 0$,
  \item or $M\iso B^{-1}\otimes L$ for some effective line bundle $L$, such that $\eta\in X_{|L|}$.
\end{enumerate}
\end{proposition}
\begin{proof}
Write $M$ as $B^{-1}\otimes L$ for some line bundle $L$. Twisting (\ref{extension}) by $B^{-1}\otimes L$ yields the exact sequence
\begin{equation*}
    0\rightarrow{H^0\paren{B^{-1}\otimes L}}\rightarrow{H^0\paren{E\otimes B^{-1}\otimes L}}\rightarrow{H^0(L)}\xrightarrow{\delta} H^1\paren{B^{-1}\otimes L}\rightarrow\cdots,
\end{equation*}
which implies
\begin{equation}\label{number of sections}
  h^0\paren{E\otimes M}=h^0(M)+\dim(\ker{\delta}).
\end{equation}
Then the proposition follows from the lemma below.
\end{proof}

\begin{lemma}\label{Kernel}
\begin{equation*}
    \ker \delta=\{s\in H^0(L) |\; D=(s)_0, \eta\in \lsp{D}\}.
\end{equation*}
\end{lemma}
\begin{proof}
Given $s\in H^0(L)$,let $D=(s)_0$. There exists an associated sequence $\ses{B(-D)}{B}{B\otimes\sshf{D}}$, which can be completed by the snake lemma as follows,

\begin{equation*}
    \xymatrix{
            &           & 0\ar[d]                  & 0\ar[d]  \\
   0 \ar[r] & \sshf{C}\ar[r]\ar@{=}[d] & F\ar[r]\ar[d] & B(-D)\ar[r]\ar[d]^{\sigma_D} & 0\\
   0 \ar[r] & \sshf{C}\ar[r] &  E\ar[r]\ar[d] &  B \ar[r]\ar[d] & 0 \\
            &           & B\otimes\sshf{D}\ar[d] \ar@{=}[r]                    &   B\otimes\sshf{D}\ar[d]\\
            &           & 0                         &    0
}
\end{equation*}
Noting that the extension class of the first row is $\delta(s)\in H^1(B^{-1}\otimes L)\iso \textrm{Ext}^1(B\otimes L^{-1}, \sshf{C})$, we have
\begin{eqnarray*}
    && \eta\in\lsp{D} \\
   &\iff & \text{the composition} \quad H^0\paren{K_C\otimes B(-D)}\rightarrow H^0(K_C\otimes B)\xrightarrow{\eta} H^1(K_C)\quad\text{is zero} \\
   &\iff& \ses{\sshf{C}}{F}{B(-D)} \quad\text{splits}  \\
   &\iff & \delta(s)=0.
\end{eqnarray*}
\end{proof}

Recall that for a ruled surface $X$, $H^2(X, \sshf{X})=0$, so a divisor $\Sigma\subset X$ is semi-regular if and only if $H^1\paren{\sshf{\Sigma}(\Sigma)}=0$. As in \cite{HA}, we denote a closed fibre of $\pi: X\rightarrow C$ by $f$. For a divisor $\alpha$ on $C$, we write $\alpha f$ for $\pi^*\alpha$ by abuse of notation.

\begin{proposition}[Semi-regularity Criterion]\label{obstruction group}
Let $\Sigma=\Gamma+\alpha f\in |\sshf{X}(1)\otimes \pi^*(B^{-1}\otimes L)|$, where $\Gamma$ is the image of a section $\sigma: C\rightarrow X$. Then $\sigma$ corresponds to
\begin{equation*}
    \ses{B\otimes L^{-1}(\alpha)}{E}{L(-\alpha)}.
\end{equation*}
And the obstruction
group $H^1\paren{\sshf{\Sigma}(\Sigma)}\iso H^0\paren{C, K_C\otimes B\otimes L^{-2}(\alpha)}^*$.
\end{proposition}
\begin{proof}
$\Gamma$ arises from some one dimensional quotient of $E$
\begin{equation*}
    \ses{N}{E}{M}.
\end{equation*}
By \cite[V, 2.6]{HA}, $\pi^* N\iso \sshf{X}(1)\otimes\sshf{X}(-\Gamma)$, which is isomorphic to $\pi^*\paren{B\otimes L^{-1}(\alpha)}$, therefore $N\iso B\otimes L^{-1}(\alpha)$. Since $\det E\iso B$, one has $M\iso L(-\alpha)$.

Assume $\alpha=\sum^m_{i=1}a_ip_i$, where $a_i\in \mathbb{N}$ and $p_i\in C$. Let $\Sigma_i=\Gamma+\sum_{j\le i}a_jf$, where $a_if$ is $\pi^{-1}(a_ip_i)$. In this notation $\Sigma_0=\Gamma, \Sigma_m=\Sigma$. Consider the exact sequence
\begin{equation*}
    \ses{\sshf{\Sigma_{i}}}{\sshf{\Sigma_{i-1}}\oplus\sshf{a_i f}}{\sshf{Z_i}},
\end{equation*}
where $Z_i$ is the scheme theoretic intersection of $\Sigma_{i-1}$ with $a_if$.
Tensoring it with $\shf{L}:=\sshf{X}(1)\otimes\pi^*(B^{-1}\otimes L)\iso\sshf{X}(\Sigma)$, one obtains
\begin{equation*}
    \cdots\rightarrow H^0(\shf{L}|_{\Sigma_{i-1}})\oplus H^0(\shf{L}|_{a_i f})\rightarrow H^0(\shf{L}|_{Z_i})\rightarrow H^1(\shf{L}|_{\Sigma_i})\rightarrow H^1(\shf{L}|_{\Sigma_{i-1}})\oplus H^1(\shf{L}|_{a_i f})\rightarrow 0.
\end{equation*}

On the one hand, denote the ideal sheaf of $\pi^{-1}(p_i)\iso\prj{1}$ by $\is{}$. For $k\ge 1$, there exists the sequence
\begin{equation*}
    \ses{\is{}^{k-1}/{\is{}^{k}}}{\sshf{X}/{\is{}^{k}}}{\sshf{X}/{\is{}^{k-1}}}.
\end{equation*}
By flatness of $\pi$ and smoothness of fibre $f$, $\is{}^{k-1}/{\is{}^{k}}\iso \sshf{\prj{1}}$. One obtains
$H^1(\shf{L}|_{kf})\iso H^1(\shf{L}|_{(k-1)f})\iso\cdots H^1(\shf{L}|_{f})\iso H^1(\prj{1}, \sshf{\prj{1}}(1))=0$. In particular, $H^1(\shf{L}|_{a_i f})=0$.

On the other, $H^0(\shf{L}|_{a_i f})\rightarrow H^0(\shf{L}|_{Z_i})$ is surjective, since its cokernel $H^1(a_if, \shf{L}(-\Gamma)|_{a_if})=0$ by a similar argument as above.

So $H^1(\shf{L}|_{\Sigma_i})\iso H^1(\shf{L}|_{\Sigma_{i-1}})$. Inductively, one gets
\begin{eqnarray*}
    H^1(\sshf{\Sigma}(\Sigma)) &\iso & H^1(\shf{L}|_{\Sigma_m})\\
    &\iso & H^1\paren{\shf{L}|_{\Gamma}}\\
    &\iso& H^1\paren{C, \sigma^*\paren{\sshf{X}(1)\otimes\pi^*\paren{B^{-1}\otimes L}}}\\
    &\iso& H^1\paren{C, B^{-1}\otimes L^2(-\alpha)}\\
    &\iso& H^0\paren{C, K_C\otimes B\otimes L^{-2}(\alpha)}^*,
\end{eqnarray*}
and we are done.
\end{proof}

\begin{remark}\label{non semi-regular locus}
The obstruction group indicates that if $2\deg{L}<g+d_0-1$, then $\sshf{X}(1)\otimes \pi^*(B^{-1}\otimes L)$ is not semi-regular. For instance, take $g=3$, $d_0=3$ and $\eta\in S^1 C$. In this case $W^0_{1, -1}\neq\emptyset$, but none of its point is semi-regular. This produces many examples of effective line bundle locus on Pic($X$) which does not contain any semi-regular line bundle.
\end{remark}

\begin{proposition}\label{irreducibility}
Assume $H^0\paren{C, B^{-1}\otimes L}=0$ and $\eta\in X_{|L|}$. Then there exists a reducible $\Sigma\in |\sshf{X}(1)\otimes\pi^*(B^{-1}\otimes L)|$ if and only if there exists a line bundle $L'\subsetneq L$, such that $\eta\in X_{|L'|}$.
\end{proposition}
\begin{proof}
Let $\Sigma=\Gamma+\alpha f\in |\sshf{X}(1)\otimes\pi^*(B^{-1}\otimes L)|$ for some effective $\alpha$ with $\deg{\alpha}\ge 1$. Then
$\sshf{X}(\Gamma)\iso \sshf{X}(1)\otimes \pi^*(B^{-1}\otimes L(-\alpha))$. Notice $H^0\paren{C, B^{-1}\otimes L(-\alpha)}=0$, so by Corollary \ref{characterization}, $X_{|L(-\alpha)|}\owns\eta$.

Conversely let $L'\iso L(-\alpha)$ for some effective $\alpha$ with $\deg{\alpha}\ge 1$ and $X_{|L'|}\owns\eta$, then $H^0\paren{\sshf{X}(1)\otimes\pi^*\paren{B^{-1}\otimes L'}}\neq 0$. Let $\Sigma\in |\sshf{X}(1)\otimes\pi^*(B^{-1}\otimes L')| $, then the reducible divisor $\Sigma+\alpha f\in|\sshf{X}(1)\otimes\pi^*(B^{-1}\otimes L)|$.
\end{proof}

To state the example below, we define the set
\begin{equation*}
    X^i_{|L|}=\bigcup_{\substack{h^0(L(-D))\ge 1\\ \deg{D}=\deg{L}-i}}\lsp{D},
\end{equation*}
for $i\ge 0$, namely the set swept out by linear spans of all degree $(\deg{L}-i)$ sub-divisors of $|L|$. The inclusion $X^{i+1}_{|L|}\subseteq X^i_{|L|}$ is clear.

\begin{corollary}\label{semi-regularity}
With the above notation, assume $H^0(C, B^{-1}\otimes L)=0$ and $\eta\in X_{|L|} \backslash X^1_{|L|}$. Then $\sshf{X}(1)\otimes\pi^*(B^{-1}\otimes L)$ is semi-regular if and only if $H^0(C, K_C\otimes B\otimes L^{-2})=0$.\qed
\end{corollary}

\subsection{Example}\mbox{}\\

Fix $d_0\ge 3$ as before. Choose positive integers $g$, $d$ satisfying
\begin{enumerate}
  \item $g+d_0+1\le 2d\le 2g$,
  \item $2d-d_0$ is a prime number.
\end{enumerate}
We can construct a ruled surface $X$ over a curve of genus $g$, such that one component of $W^0_{1, d-d_0}(X)$ of dimension $(2d-d_0-g+1)$ satisfies the rules proposed at the beginning of \S 3.
\vspace{0.3cm}

To be precise, take a \emph{general} curve $C$ of genus $g$  (hence all Brill-Noether loci $W^r_d(C)$ involved below have the expected dimensions $\rho=g-(r+1)(g-d+r)$ ). Recall for any $B$ of degree $d_0$, $|K_C\otimes B|$ induces an embedding $\varphi: C\hookrightarrow\prj{N}$ with $N=g+d_0-2$. Then
\vspace{0.3cm}

\textbf{Claim A}: There exist $B\in\textrm{Pic}^{d_0}(C)$ and $L\in \textrm{Pic}^d(C)$ satisfying the conditions:
\begin{enumerate}
 \item[A.1] $h^0(B^{-1}\otimes L)=0$,
 \item[A.2] $h^0(L)=2$,
 \item[A.3] $L$ is base point free,
 \item[A.4] $h^0\paren{K_C\otimes B\otimes L^{-2}(p)}=0$ for all $p\in C$,
 \item[A.5] Let $\Lambda=\cap_{D\in |L|} \lsp{D}$, then $\Lambda\backslash X^2_{|L|}\neq\emptyset$.
 \end{enumerate}

\vspace{0.3cm}

\textbf{Claim B}: Fix  $B\in\textrm{Pic}^{d_0}(C)$. Then for general points $q_1, \cdots, q_d$ on $C$, the following conditions hold
\begin{itemize}
 \item[B.1] $h^0(B^{-1}(q_1+\cdots+q_{d}))=0$,
 \item[B.2] $h^0\paren{\sshf{C}(q_1+\cdots+q_{d})}=1$,
 \item[B.3] $h^0(K_C\otimes B(-2q_1-\cdots-2q_{d}))=0$.
\end{itemize}
Moreover if $\eta\in\Lambda\backslash X^2_{|L|}$ is chosen properly, there is a nonempty (locally closed) subset $W$ of the $d$-th symmetric product of $C$ such that for $q_1+\cdots+q_d\in W$,
\begin{itemize}
  \item [B.4] $\eta\in\lsp{q_1+\cdots+q_d}$, but $\eta\notin\lsp{q_1+\cdots+\hat{q_i}+\cdots+q_d}$ for all $i$.
\end{itemize}
Here the notation $\hat{q_i}$ means omit $q_i$.
\vspace{0.3cm}

Granting the claims for the moment, the quadruple $(C, B, \eta, L)$ chosen above determines a ruled surface $\pi: X=\mathbb{P}(E)\rightarrow C$. Let $\shf{L} =\sshf{X}(1)\otimes \pi^*\paren{B^{-1}\otimes L}$. Then $h^0(X, \shf{L})=2$ by A.1, A.2, A.5 and (\ref{number of sections}). Moreover $\shf{L}$ is semi-regular. In fact, writing a divisor $\Sigma\in |\shf{L}|$ as $\Gamma+\alpha f$ as in Proposition \ref{obstruction group}, one has
\begin{equation*}
    \sshf{X}(\Gamma)\iso \sshf{X}(1)\otimes\pi^*\paren{B^{-1}\otimes L(-\alpha)}.
\end{equation*}
Since $h^0\paren{B^{-1}\otimes L(-\alpha)}=0$, $\eta\in X_{|L(-\alpha)|}$ by Proposition \ref{characterization}. A.5 thereby implies $\deg{\alpha}\le 1$. Therefore one can use Proposition \ref{obstruction group} and A.4 to deduce the semi-regularity of $\shf{L}$. So $\shf{L}\in W^0_{1, d-d_0}(X)$.

On the other hand, for points $q_1, \cdots, q_d$ satisfying B.1-B.4, $\shf{L}':=\sshf{X}(1)\otimes \pi^*B^{-1}(q_d+\cdots+q_d)$ is semi-regular (B.1, B.3, B.4 and Corollary \ref{semi-regularity}), and has one dimensional sections (B.1, B.2, B.4). Because of the irreducibility of $W^0_d(C)$, $\shf{L}'$ specializes to the component $\Omega\subset W^0_{1, d-d_0}(X)$ which contains $\shf{L}$. Therefore, the induced Abel-Jacobi map
$\varphi: \textrm{Div} X\rightarrow \Omega$ is a birational morphism, and hence
\begin{eqnarray*}
    \dim \Omega &=&R\\
                &=&\chi(\shf{L})-1+q(X) \hspace{3cm}\text{by}\;(\ref{expected dimension})\\
                &=& 2d-d_0+1-g.         \hspace{3cm}\text{by the Riemann-Roch for}\;\pi_*\shf{L}
\end{eqnarray*}
For $\shf{L}$, $b=r+1-\chi(\shf{L})=2g+d_0-2d$, as no $H^2$ involved. We apply Corollary \ref{Riemann-Kempf} to get that the multiplicity $\mu$ of $\Omega$ at $[\shf{L}]$ equals $2g+d_0-2d$, which is larger than 1. Our main theorem asserts that $\Omega$ has at worst rational singularities.

\subsection{Proofs of Claims}\mbox{}\\

The proof for A.1-A.4 proceeds by dimension counting, while for A.5, namely $X^2_{|L|}\neq\Lambda$, we need a delicate computation of cohomologies.

\begin{proof}[Proof of Claim A]
We first look for a pair $(B, L)$ with the constraints A.1-A.4 all satisfied. For this purpose, we define the morphisms
\begin{eqnarray*}
     m_B:     & \textrm{Pic}^d(C)\rightarrow \textrm{Pic}^{d+d_0}(C)                      & \quad M\mapsto M\otimes B,\\
    \gamma:   & \textrm{Pic}^{d}(C)\rightarrow \textrm{Pic}^{2d}(C)                       &   \quad M\mapsto M^{\otimes 2},\\
    \alpha:   &  W^1_{d-1}\times C\rightarrow \textrm{Pic}^d(C)                           & \quad (M, p)\mapsto M(p),\\
    \beta:    &  W^{2d-g-d_0}_{2d-d_0-1}\times C\rightarrow\textrm{Pic}^{2d-d_0}(C)       & \quad (M, p)\mapsto M(p).
\end{eqnarray*}

A.1 and A.2 are to say that $L\in W^1_d\backslash m_B(W^0_{d-d_0})$; A.3 is to say $L\notin \textrm{im}(\alpha)$. By Riemman-Roch and Serre duality, A.4 is translated to the condition:
\begin{equation*}
   h^0(B^{-1}\otimes L^2(-p))=2d-g-d_0 \quad \text{for all}\; p\in C,
\end{equation*}
which is in turn equivalent to
\begin{equation*}
    B^{-1}\otimes L^2\quad\text{is base point free with}\;\; h^0(B^{-1}\otimes L^2)=2d-g-d_0+1.
\end{equation*}

So $m_{B^{-1}}\circ\gamma(L)\notin W^{2d-g-d_0+1}_{2d-d_0}\cup\textrm{Im}(\beta)$.

Note $\dim{\textrm{Im}(\alpha)}<\dim W^1_d$, so to attain A.1-A.4 simultaneously, it suffices to choose $B$ such that
\begin{equation}\label{desired containment relation}
    W^1_{d}\nsubseteq \gamma^{-1}\paren{m_B \paren{W^{2d-g-d_0+1}_{2d-d_0}\cup\textrm{Im}\beta}}\cup m_B\paren{W^0_{d-d_0}},
\end{equation}
see the diagram
$$\xymatrix{
                                  &         W^{2d-g-d_0+1}_{2d-d_0}\ar@{^{(}->}[d]         &                   &               W^0_{d-d_0}\ar[d]^{m_B}  &  \\
W^{2d-g-d_0}_{2d-d_0-1}\times C\ar[r]^{\beta}    &     \textrm{Pic}^{2d-d_0}(C)\ar[r]^{m_B}               &  \textrm{Pic}^{2d}(C) & \textrm{Pic}^{d}(C)\ar[l]_{\gamma} & W^{1}_{d-1}\times C\ar[l]_{\alpha}
}$$
\vspace{0.3cm}

Regarding $\textrm{Pic}^{d}(C)$ as a homogeneous space with the group $\textrm{Pic}^0(C)$ acting in the obvious way, we apply the Kleiman's transversality \cite[Theorem 2]{KL2} to get that, for generic $B\in\textrm{Pic}^{d_0}(C)$, the intersection of $\gamma^{-1}\paren{m_B \paren{W^{2d-g-d_0+1}_{2d-d_0}\cup\textrm{Im}\beta}}\cup m_B\paren{W^0_{d-d_0}}$ with $W^1_{d}$ has the expected dimension, which is less than $\dim W^1_d=2d-g-2$. Thereby (\ref{desired containment relation}) holds for generic $B\in\text{Pic}^d_0(C)$, consequently $L$ has $(2d-g-2)$ dimensional freedom of choice for a fixed $B$.
\vspace{0.3cm}

Then for A.5, we first show that $\Lambda\iso\prj{2d-g-d_0}$.
\vspace{0.3cm}

By Proposition \ref{chariterion of resolution of singularities}, $\Lambda=\lsp{D_1}\cap\lsp{D_2}$ for any distinct $D_1, D_2\in |L|$. Observe that\\
$H^0(\is{{\lsp{D_1}\cap\lsp{D_2}}/\prj{N}}(1))$ is given by
 $\text{Im} (H^0(K_C\otimes B(-D_1))\oplus H^0(K_C\otimes B(-D_2))\rightarrow H^0(K_C\otimes B)))$. To calculate it, we use the base point free pencil trick for $L$. The short exact sequence
\begin{equation*}
    \ses{\sshf{C}(-D_1-D_2)}{\sshf{C}(-D_1)\oplus\sshf{C}(-D_2)}{\sshf{C}}
\end{equation*}
yields
\begin{eqnarray*}
    &&\dim\text{Im} (H^0(K_C\otimes B(-D_1))\oplus H^0(K_C\otimes B(-D_2))\rightarrow H^0(K_C\otimes B))\\
    &=& 2h^0(K_C\otimes B\otimes L^{-1})-h^0(K_C\otimes B\otimes L^{-2})\\
    &=& 2(g+d_0-d-1),
\end{eqnarray*}
which implies $\lsp{D_1}\cap\lsp{D_2}\iso\prj{2d-g-d_0}$.
\vspace{0.3cm}

To prove $\Lambda\backslash X^2_{|L|}\neq\emptyset$, it suffices to show  for a \emph{general} degree $d-2$ divisor $Z$ with $h^0(L(-Z))>0$,  $\dim(\lsp{Z}\cap\Lambda)= 2d-g-d_0-2$.
\vspace{0.3cm}

To this end, we fix $D_0\in |L|$, assume $Z+p+q=D\in |L|$ for some points $p, q\in C$ and that $D\neq D_0$. $\dim\lsp{Z}\cap\Lambda=2d-g-d_0-2$ if and only if the image of the diagonal map
$$\xymatrix{
H^0(K_C\otimes B(-D_0)) \ar[d] \ar@{^{(}->}[r] \ar@{-->}[dr] & H^0(K_C\otimes B)\ar[d]\\
H^0(K_C\otimes B(-D_0)|_Z)\ar@{^{(}->}[r] & H^0(K_C\otimes B|_Z)}$$
is $(g+d_0-d-1)$ dimensional as a vector space.

From the exact sequence
\begin{equation*}
    0\rightarrow H^0(K_C\otimes B (-D_0-Z))\rightarrow H^0(K_C\otimes B(-D_0))\rightarrow H^0(K_C\otimes B(-D_0)|_Z),
\end{equation*}
and the fact $h^0(K_C\otimes B(-D_0))=g+d_0-d-1$, we see this happens if and only if $h^0(K_C\otimes B(-D_0-Z))=h^0\paren{K_C\otimes B\otimes L^{-2}(p+q)}=0$.

By A.4 and its reformulation, $B^{-1}\otimes L^2$ is base point free and $h^0\paren{K_C\otimes B\otimes L^{-2}(p)}=0$. Therefore $h^0\paren{K_C\otimes B\otimes L^{-2}(p+q)}=0$ if and only if $|B^{-1}\otimes L^2|$ separates $p$ and $q$.

Denote the image of the map $C\xrightarrow{|B^{-1}\otimes L^2|}\prj{N'}$ as $C'$. Obviously $C'$ is not $\prj{1}$. Since $\deg\paren{B^{-1}\otimes L^2}=2d-d_0$ is prime, the induced map $C\rightarrow C'$ cannot be a finite morphism of degree $\ge 2$, and hence $C\rightarrow C'$ is birational. The number of pairs $(p, q)$ which $B^{-1}\otimes L^2$ cannot separate is finite.
\end{proof}
\vspace{0.3cm}

\begin{proof}[Proof of Claim B (Sketch)]
We first pick two distinct points $q_1, q_2$ such that $h^0(K_C\otimes B(-2q_1-2q_2))=h^0(K_C\otimes B)-4$, then proceed by induction on the number of points.
Suppose $q_1, \cdots, q_i$ for $2\le i\le d-1$ have been picked with the conditions
\begin{itemize}
  \item $h^0(B^{-1}\paren{q_1+\cdots+q_i)}=0$,
  \item $h^0\paren{\sshf{C}(q_1+\cdots+q_i)}=1$,
  \item $h^0\paren{K_C\otimes B(-2q_1-\cdots-2q_i)}=\max{\{h^0(K_C\otimes B)-2i, 0\}}$.
\end{itemize}
If $q_{i+1}$ is chosen by avoiding $C\cap \lsp{q_1+\cdots+q_i}$ and any inflectionary points (cf. \cite[p. 37]{ACGH}) of $|K_C(-q_1-\cdots-q_i)|$ and $|K_C\otimes B(-2q_1-\cdots-2q_i)|$, which are finite, then the conditions still hold for $i+1$. In this process, we need the assumption $g+d_0+1\le 2d\le 2g$ to guarantee that neither $|K_C(-q_1-\cdots-q_i)|$ nor $|K_C\otimes B(-2q_1-\cdots-2q_i)|$ is empty.

The second part of the claim is quite obvious and we omit its proof.
\end{proof}

\section{Appendix}
In the section, we review the construction of $X_{|L|}$, the variety swept out by all linear spans of divisors in $|L|$ on a curve $C$ with respect to an embedding $C\subset\prj{N}$.
\vspace{0.3cm}

Assume $A$ is a very ample line bundle on the curve $C$ with genus $g\ge 1$. Let $V$ denote $H^0(C, A)$. Given an effective line bundle $L$ of degree $d$ with $\dim |L|=r$. Denote the two projections from $C\times |L|$ by $p$ and $q$. There exists the sequence
\begin{equation*}
   \ses{p^*L^{-1}\otimes q^*\sshf{|L|}(-1)}{\sshf{C\times |L|}}{\sshf{\mathscr{D}}},
\end{equation*}
where $\mathscr{D}$ is the universal divisor, see (\ref{universal family}).

Applying $q_*(p^* A\otimes \_)$ to the above, we get the exact sequence
\begin{eqnarray*}\label{segre}
    &&0\rightarrow {H^0\paren{C, A\otimes L^{-1}}\otimes \sshf{|L|}(-1)}\rightarrow{V\otimes\sshf{|L|}}\rightarrow{q_*\paren{p^* A\otimes \sshf{D}}}\rightarrow
    {H^1\paren{C, A\otimes L^{-1}}\otimes\sshf{|L|}(-1)}\\
    &&\rightarrow H^1\paren{C, A}\otimes\sshf{|L|}\rightarrow 0.
\end{eqnarray*}
where each term is locally free. Consequently
\begin{equation}\label{Segre}
    \ses{H^0\paren{C, A\otimes L^{-1}}\otimes \sshf{|L|}(-1)}{V\otimes\sshf{|L|}}{\shf{Q}},
\end{equation}
where the cokernel $\shf{Q}$ is locally free.
$(\ref{Segre})$ yields the diagram
$$\xymatrix{
\mathbb{P}(\shf{Q})\ar@{^{(}->}[r] & \mathbb{P}(V\otimes\sshf{|L|}) \ar[d]^{\pi} \ar[r]^{\phi} & \mathbb{P}(V)\\
 & |L| & &}$$

By abuse of notation we denote the induced maps from $\mathbb{P}(\shf{Q})$ to $\mathbb{P}(V)$ and $|L|$ by $\phi$ and $\pi$ respectively. $X_{|L|}$ is defined as the scheme theoretic image of $\phi$. Geometrically, $X_{|L|}$ is the union of all linear spans of divisors $D\in |L|$.

\begin{proposition}\label{chariterion of resolution of singularities}
With notations as above. We write $A=K_C\otimes B$ and assume $\deg{B}=d_0\ge 3$. Then any fibre of $\phi:\mathbb{P}(\shf{Q})\rightarrow X_{|L|}$ over a closed point is a projective space.
$\phi$ is birational if and only if $r\le h^0(A\otimes L^{-1})$. Furthermore, if $r<h^0(A\otimes L^{-1})$, then $X_{|L|}$ is Cohen-Macaulay.
\end{proposition}
\begin{proof}
First note that as in \S 3.1, every $x\in\mathbb{P}(V)$ determines an extension
\begin{equation*}
    \ses{\sshf{C}}{E}{B},
\end{equation*}
unique up to isomorphism. By Lemma \ref{Kernel}, $\pi\circ\phi^{-1}(x)\iso \mathbb{P}((\ker\delta)^{\vee})$, which we denote by $P$ for short.

We claim that $\pi: \phi^{-1}(x)\rightarrow  P$ is an isomorphism. In fact, let $V\rightarrow L_0$ be the one dimensional quotient representing $x$. Then one has the commutative diagram ($P$ is the locus in $|L|$ where $V\otimes\sshf{|L|}\rightarrow L_0\otimes\sshf{|L|}$ factors through $Q$)
$$\xymatrix{
V\otimes_k\sshf{P}\ar@{->>}[d] \ar@{->>}[r] & L_0\otimes_k\sshf{P}\\
Q|_P \ar@{->>}[ru] &}$$
whose oblique map, by \cite[II 7.12]{HA}, induces the commutative diagram
$$\xymatrix{
 & \mathbb{P}(\shf{Q})\ar[d]^{\pi}\\
P \ar[ru]^{\sigma}\ar@{^{(}->}[r] & |L|}$$
i.e. $\pi\circ\sigma=\text{id}_P$. On the other hand, $\sigma(P)=\mathbb{P}(L_0\otimes_k\sshf{P})=\phi^{-1}(x)$. This establishes the isomorphism $\pi: \phi^{-1}(x)\rightarrow P$. Hence $\phi^{-1}(x)$ is a projective space.

In our case, $\phi$ is birational if and only if it is generically finite. Since
$\phi$ is induced by the tautological line bundle $\sshf{\mathbb{P}(\shf{Q})}(1)$, $\phi$ is generically finite if and only if
$(c_1(\sshf{\mathbb{P}(\shf{Q})}(1)))^{\dim{\mathbb{P}(\shf{Q})}}=s_r\paren{\shf{Q}^{\vee}}>0$.

Let $H$ be the hyperplane section class on $|L|$ and $s_t(\_)$ and $c_t(\_)$ denote the Segre and Chern polynomials respectively. Then by (\ref{Segre}),
\begin{eqnarray*}
    s_t\paren{\shf{Q}^{\vee}}&=&c_t\paren{H^0\paren{C, A\otimes L^{-1}}^*\otimes \sshf{|L|}(1)}\\
    &=& (1+tH)^{h^0\paren{C, A\otimes L^{-1}}}.
\end{eqnarray*}
It follows that $s_r(\shf{Q}^{\vee})={h^0\paren{A\otimes L^{-1}}\choose r}H^r$. So $s_r(\shf{Q}^{\vee})>0$ if and only if $r\le h^0 \paren{A\otimes L^{-1}}$.

When $r<h^0(A\otimes L^{-1})$, consider the map of vector bundles
\begin{equation*}
    \xi: H^0\paren{A\otimes L^{-1}}\otimes\sshf{\mathbb{P}(V)}(-1)\rightarrow H^0(L)^*\otimes\sshf{\mathbb{P}(V)}.
\end{equation*}

Let $X_{|L|}=\{x\in\mathbb{P}(V)| \:\rk{\xi_x}\le r\}$ with the determinantal variety structure. Then
\begin{eqnarray*}
  \dim{X_{|L|}} &=&\dim\mathbb{P}(\shf{Q})\\
   &=& r+\rk{\shf{Q}}-1 \\
   &=& r+h^0(A)-h^0\paren{A\otimes L^{-1}}-1\\
   &=& \dim{\mathbb{P}(V)}-\paren{h^0\paren{A\otimes L^{-1}}-r}\paren{h^0(L)-r},
\end{eqnarray*}
which is the expected dimension, therefore $X_{|L|}$ is Cohen-Macaulay (cf. \cite{ACGH} p. 84).
\end{proof}

\bibliography{Onthesingularitiesofeffectivelocioflinebundles}{}
\bibliographystyle{amsplain}
\end{document}